\documentclass[12pt]{article}

\usepackage{amsthm, amsmath, amssymb,graphicx}

\numberwithin{equation}{section}

\theoremstyle{plain}	     
\newtheorem{thm}{Theorem}[section] 
\newtheorem{cor}[thm]{Corollary}

\theoremstyle{definition}

\theoremstyle{remark} 
\newtheorem{rem}[thm]{Remark}
	
\newcommand{\vp}{\varphi}

\newcommand{\slem}{\operatorname{sl}}

\begin{document}

\title{Applications of generalized trigonometric functions
with two parameters
\footnote{The work of S.\,Takeuchi was supported by JSPS KAKENHI Grant Number 17K05336.}}
\author{Hiroyuki Kobayashi and Shingo Takeuchi \\
Department of Mathematical Sciences\\
Shibaura Institute of Technology
\thanks{307 Fukasaku, Minuma-ku,
Saitama-shi, Saitama 337-8570, Japan. \endgraf
{\it E-mail address\/}: shingo@shibaura-it.ac.jp \endgraf
}}
\date{}

\maketitle

\begin{abstract}
Generalized trigonometric functions (GTFs) are simple generalization of the classical 
trigonometric functions. GTFs are deeply related to the $p$-Laplacian, which is known 
as a typical nonlinear differential operator, and there are a lot of works on GTFs 
concerning the $p$-Laplacian. However, few applications to differential equations unrelated 
to the $p$-Laplacian are known.
We will apply GTFs with two parameters to
nonlinear nonlocal boundary value problems without $p$-Laplacian.
Moreover, we will give integral formulas for the functions, e.g. Wallis-type formulas, 
and apply the formulas to the lemniscate function and the lemniscate constant. 
\end{abstract}


\section{Introduction}

Let $p,\ q \in (1,\infty)$ be any constants. 
We define $\sin_{p,q}{x}$ by
the inverse function of 
\[
\sin_{p,q}^{-1}{x}:=\int_0^x \frac{dt}{(1-t^q)^{1/p}}, \quad 0 \leq x \leq 1,
\]
and $\pi_{p,q}$ by
\[ \pi_{p,q}:=2\sin_{p,q}^{-1}{1}=2\int_0^1 \frac{dt}{(1-t^q)^{1/p}}
=\frac2q B\left(\frac{1}{p^*},\frac1q\right),\]
where $p^*:=p/(p-1)$ and $B$ denotes the beta function
\[ B(x,y)=\int_0^1 t^{x-1}(1-t)^{y-1}\,dt, \quad x,\ y>0. \]
Clearly, the function $\sin_{p,q}{x}$ is increasing in $[0,\pi_{p,q}/2]$ onto $[0,1]$.
Since $\sin_{p,q}{x} \in C^1(0,\pi_{p,q}/2)$,
we define $\cos_{p,q}{x}$ by $\cos_{p,q}{x}:=(\sin_{p,q}{x})'$.
In case $p=q$, we denote $\sin_{p,p}{x}$, $\cos_{p,p}{x}$ and $\pi_{p,p}$
briefly by $\sin_p{x}$, $\cos_p{x}$ and $\pi_p$, respectively.
It is obvious that $\sin_{2}{x},\ \cos_{2}{x}$ 
and $\pi_{2}$ are reduced to the ordinary $\sin{x},\ \cos{x}$ and $\pi$,
respectively. This is the reason why these functions and the constant are called
\textit{generalized trigonometric functions} (GTFs) with parameter $(p,q)$
and \textit{the generalized $\pi$}, respectively. 
As the trigonometric functions satisfy $\cos^2{x}+\sin^2{x}=1$,
so it is shown that for $x \in [0,\pi_{p,q}/2]$
\[ \cos_{p,q}^p{x}+\sin_{p,q}^q{x}=1.\]
Moreover, we see that
\begin{equation}
\label{eq:dcos}
(\cos_{p,q}^{p-1}{x})'=\frac{(p-1)q}{p}\sin_{p,q}^{q-1}{x},
\end{equation}
which implies that
$u=\sin_{p,q}{x}$ satisfies the nonlinear differential equation with the $p$-Laplacian:
\[-(|u'|^{p-2}u')'=\frac{(p-1)q}{p}|u|^{q-2}u.\]
In case $p=q=2$, this is reduced to the simple harmonic oscillator equation 
$-u''=u$ for $u=\sin{x}$.

E.\,Lundberg originally introduced GTFs in 1879;
see \cite{LP2004} for details. After his work,
there are a lot of works in which 
GTFs and related functions 
are used to study properties as functions and problems of
existence, bifurcation and oscillation of solutions of differential equations.
See \cite{Bu1964,BE2012,DoR2005,EGL2012,E1981,LE2011,Li1995,LP2003,LP2004,
Ne2016,Sh1959,T2016b}
for general properties as functions;
\cite{DEM1989,DoR2005,DM1999,LE2011,N1995,T2012,T2016b} for applications to differential equations
involving the $p$-Laplacian;
\cite{BBCDG2006,BE2012,BL2011,EGL2012,EGL2014,LE2011,T2014} for basis properties for sequences 
of these functions;
\cite{BYpreprint,KT2017,T2012,T2016a,T2016c,Tpreprint,YH2015} 
for elliptic integrals defiined by GTFs.
However, few fundamental formulas of GTFs, including the addition theorem, and few applications 
to differential equations unrelated to the $p$-Laplacian are known, though they are 
simple generalization of the classical trigonometric functions.

In this paper, for GTFs with two parameters, 
we will give applications to differential equations (without the $p$-Laplacian) and integral formulas.
In Section 2, we will solve the nonlinear nonlocal boundary value problem:
\begin{equation}
\label{prob:1}
\vp'-(\vp')^2+\vp\vp''+\frac{2}{H}\int_0^H(\vp'(t))^2\,dt=0,\quad
\vp(0)=\vp(H)=0.
\end{equation}
This problem was studied in C.\,Cao et al \cite{CINT2015}
to investigate the self-similar 
blowup for the inviscid primitive equations of oceanic and atmospheric dynamics.
They showed the existence of positive solutions, but
gave no expression of the solutions.
Using GTFs, we will be able to express 
all positive solutions of problems including \eqref{prob:1} in terms of GTFs 
with two parameters. In particular, all the positive solutions of problem \eqref{prob:1} 
will be given as
\[\vp=\frac{2H}{(2-r)\pi_{r}}\cos_{r}^{r-1}{\left(\frac{\pi_{r}}{2H}x\right)}
\sin_{r}{\left(\frac{\pi_{r}}{2H}x\right)},\]
where 
\[r:=\left(\frac{1}{2}+\frac{1}{4\sqrt{m^2+1/4}}\right)^{-1} \in (1,2),\]
and $m>0$ is a free parameter.
In Section 3, we will construct integral formulas for GTFs with two parameters, e.g.
\[\int_0^x \sin_{p,q}^k{t} \cos_{p,q}^\ell{t}\,dt
=\frac{1}{k+1}\sin_{p,q}^{k+1}{x}
F\left(\frac{k+1}{q},\frac{1-\ell}{p};1+\frac{k+1}{q};\sin_{p,q}^q{x}\right),\]
\[\int_0^{\pi_{p,q}/2} \sin_{p,q}^k{t} \cos_{p,q}^\ell{t}\,dt
=\frac{1}{q}B\left(\frac{k+1}{q},1+\frac{\ell-1}{p}\right)\]
for $k>-1$ and $\ell>1-p$. Here, $F(a,b,c;x)$ is the Gaussian hypergeometric functions:
\[F(a,b;c;x):=\sum_{n=0}^\infty \frac{(a)_n(b)_n}{(c)_n}\frac{x^n}{n!}, \quad |x|<1,\]
where $(a)_n:=a(a+1)(a+2)\cdots (a+n-1)$ if $n \geq 1$ and $(a)_0:=1$.
We can find the former formula only for $p=q,\ k=1$ and $\ell=0$ 
in \cite[Proposition 2.5]{BE2012} and \cite[Proposition 2.3]{LE2011};
the latter formula only for $p=q$ in \cite[Proposition 3.1]{BE2012} 
and \cite[Proposition 2.4]{LE2011}. However, there seems to be no literature which deals with 
case $p \neq q$.
Moreover, we recall Wallis' formulas:
\[\int_0^{\pi/2} \sin^{2n}{t}\,dt
=\int_0^{\pi/2} \cos^{2n}{t}\,dt
=\frac12\cdot\frac34\cdot\frac56\cdot\cdots\cdot\frac{2n-1}{2n}\cdot\frac{\pi}{2},\]
\[\int_0^{\pi/2} \sin^{2n+1}{t}\,dt
=\int_0^{\pi/2} \cos^{2n+1}{t}\,dt
=\frac23\cdot\frac45\cdot\frac67\cdot\cdots\cdot\frac{2n}{2n+1}.\]
It is natural to try to obtain Wallis-type formulas for GTFs. We will give
\[\int_0^{\pi_{p,q}/2} \sin_{p,q}^{qn+r}{t}\,dt
=\frac{u(1/u)_n}{q(1/p^*+1/u)_n}\frac{\pi_{p,u}}{2}, \quad \frac{1}{u}:=\frac{r+1}{q}\]
for $r \in (-1,q-1]$ and 
\[\int_0^{\pi_{p,q}/2} \cos_{p,q}^{pn+r}{t}\,dt
=\frac{(1/v)_n}{(1/v+1/q)_n}\frac{\pi_{v^*,q}}{2}, \quad \frac{1}{v}:=\frac{r+p-1}{p}\]
for $r \in (1-p,1]$.
In particular, the former formula will be applied to obtain
Wallis-type formulas for the classical lemniscate function $\slem{x}=\sin_{2,4}{x}$, 
including
\[\int_0^{\varpi/2} \slem^{4n}{t}\,dt
=\frac{1}{3}\cdot\frac{5}{7}\cdot\frac{9}{11}\cdot\cdots\cdot\frac{4n-3}{4n-1}\cdot\frac{\varpi}{2},\]
where $\varpi=\pi_{2,4}=2.6220\ldots$ is the lemniscate constant.
It should be noted that these integrals for $\sin_{p,q}$ and $\cos_{p,q}$ 
are not necessarily equal even if $p=q$.   
Also, we have known the product formula for $\pi$:
\[\frac{\pi}{2}
=\prod_{n=1}^\infty \left(1-\frac{1}{4n^2}\right)^{-1},\]
which immediately follows from 
the infinite product formula of the sine function (see \cite[Theorem 1.2.2]{AAR1999}). 
This proof does not work for the product formula for $\pi_{p,q}$ if $p \neq q$
(in case $p=q$, the proof works well since $\pi_{p}/2=(\pi/p)/\sin{(\pi/p)}$).
However, applying our Wallis-type formulas for $\sin_{p,q}$, we will be able to show 
\[\frac{\pi_{p,q}}{2}
=\prod_{n=1}^\infty \left(1-\frac{1}{pn(qn+1-q/p)}\right)^{-1},\]
which yields, e.g.
\[\frac{\varpi}{2}
=\prod_{n=1}^\infty \left(1-\frac{1}{2n(4n-1)}\right)^{-1}.\]

\section{Applications to ODEs}

In this section, we will apply GTFs with two parameters to the nonlinear nonlocal boundary
value problem:
\begin{equation}
\label{eq:nbvp}
\vp'-(\vp')^2+\vp\vp''+\frac{2}{H}\int_0^H(\vp'(t))^2\,dt=0,\quad
\vp(0)=\vp(H)=0.
\end{equation}

The following theorem gives an expression of solutions to more general 
problem than \eqref{eq:nbvp}.

\begin{thm}
\label{thm:CINT}
Let $H>0$ and $p,\ q \in (1,\infty)$. Then,
the positive solution of the boundary value problem
\begin{equation}
\label{eq:bvp}
(p-q)u'-pq(u')^2+(p+q)uu''+1=0,\quad
u(0)=u(H)=0,
\end{equation}
is 
\begin{equation}
\label{eq:sol1}
u=\frac{2H}{q\pi_{p^*,q}}\cos_{p^*,q}^{p^*-1}{\left(\frac{\pi_{p^*,q}}{2H}x\right)}
\sin_{p^*,q}{\left(\frac{\pi_{p^*,q}}{2H}x\right)}.
\end{equation}
\end{thm}

\begin{proof}
Let $u>0$ and $v:=u'$. We have
\[\frac{1}{u}\frac{du}{dv}
=\frac{(p+q)v}{pqv^2-(p-q)v-1}
=\frac{1/p}{v+1/p}+\frac{1/q}{v-1/q}.\]
Integrating the both-sides, we obtain general integral curves in the phase
plane for $(u,v)$
\begin{equation}
\label{eq:vppsi}
u=C\left|v+\frac{1}{p}\right|^{1/p}\left|v-\frac{1}{q}\right|^{1/q}
\end{equation}
with some constant $C>0$. 
We have to require that for each $C>0$ the above curve yields a solution satisfying
$(u,v)|_{x=0}=(0,1/q)$ and $(u,v)|_{x=H}=(0,-1/p)$.
Then, $-1/p \leq v \leq 1/q$, and it follows
from \eqref{eq:bvp} and \eqref{eq:vppsi} that 
\begin{equation}
\label{eq:psi'}
\frac{dv}{dx}=-\frac{pq}{C(p+q)}
\left(v+\frac{1}{p}\right)^{1/p^*}\left(\frac{1}{q}-v\right)^{1/q^*}.
\end{equation}
Thus,
\[x(v)=\int_{1/q}^v \frac{dx}{dv}\,dv
=C\left(\frac{1}{p}+\frac{1}{q}\right)
\int_v^{1/q} \left(v+\frac{1}{p}\right)^{-1/p^*}
\left(\frac{1}{q}-v\right)^{-1/q^*}\,dv.\]
Since $x(-1/p)=H$,
\begin{align*} 
H&=C \left(\frac{1}{p}+\frac{1}{q}\right) \int_{-1/p}^{1/q}
\left(v+\frac{1}{p}\right)^{-1/p^*}
\left(\frac{1}{q}-v\right)^{-1/q^*}
\,dv\\
&=C \left(\frac{1}{p}+\frac{1}{q}\right)^{1/p+1/q} 
B\left(\frac{1}{q},\frac{1}{p}\right);
\end{align*}
that is,
\begin{equation}
\label{eq:c}
C=\frac{2H}{p(1/p+1/q)^{1/p+1/q}\pi_{q^*,p}}.
\end{equation}

For this $C$, we seek the solution $v$ of \eqref{eq:psi'}
with $v(0)=1/q$ and $v(H)=-1/p$.
Setting $v+1/p=(1/p+1/q)w^p$ in \eqref{eq:psi'},
we have
\begin{equation*}
\frac{dw}{dx}=-\frac{\pi_{q^*,p}}{2H}
(1-w^p)^{1/q^*}.
\end{equation*}
Since $w(0)=1$, we obtain
\[\int_1^w \frac{dw}{(1-w^p)^{1/q^*}}=-\frac{\pi_{q^*,p}}{2H}x,
\quad x \in [0,H],\]
that is,
\[w=\sin_{q^*,p}{\frac{\pi_{q^*,p}(H-x)}{2H}}.\] 
It follows from \eqref{eq:sinpq} in Appendix that 
\[w=\cos_{p^*,q}^{p^*-1}{\left( \frac{\pi_{p^*,q}}{2H}x \right)}.\] 
Therefore, 
\begin{equation}
\label{eq:psi3}
v=-\frac{1}{p}+\left(\frac{1}{p}+\frac{1}{q}\right)
\cos_{p^*,q}^{p^*}{\left( \frac{\pi_{p^*,q}}{2H}x \right)}.
\end{equation}
Substituting \eqref{eq:c} and \eqref{eq:psi3} into \eqref{eq:vppsi} and using $p\pi_{q^*,p}=q\pi_{p^*,q}$, we have 
\[u=\frac{2H}{q\pi_{p^*,q}}\cos_{p^*,q}^{p^*-1}{\left(\frac{\pi_{p^*,q}}{2H}x\right)}
\sin_{p^*,q}{\left(\frac{\pi_{p^*,q}}{2H}x\right)}.\]
Thus, we conclude \eqref{eq:sol1}.
\end{proof}

According to Theorem \ref{thm:CINT}, 
it is possible to give an explicit expression of the solution of \eqref{eq:nbvp}
in terms of GTFs.

\begin{cor}
\label{cor:nbvp}
The set of all positive solutions of \eqref{eq:nbvp} is 
\[\left\{\vp:\vp=2\sqrt{m^2+\frac14}\,u_m,\ m>0\right\},\] 
where $u_m$ is the positive solution \eqref{eq:sol1} of \eqref{eq:bvp} with
\begin{equation}
\label{eq:pq}
p^*=q=r:=\left(\frac{1}{2}+\frac{1}{4\sqrt{m^2+1/4}}\right)^{-1} \in (1,2).
\end{equation}
\end{cor}

\begin{proof}
Let $m>0$ be a parameter.
We consider instead of \eqref{eq:nbvp} the nonlinear boundary value problem
\begin{equation}
\label{eq:nbvp'}
\vp'-(\vp')^2+\vp\vp''+m^2=0,\quad
\vp(0)=\vp(H)=0.
\end{equation}
In case \eqref{eq:nbvp'} has a solution, which we denote by $\vp_m$,
then $\vp_m$ is nontrivial, i.e. nonconstant, because $m>0$. 
Moreover, integrating \eqref{eq:nbvp'} yields
\[m^2=\frac{2}{H}\int_0^H (\vp_m'(x))^2\,dx.\]
Consequently, $\vp_m$ is also a nontrivial solution of \eqref{eq:nbvp}. Therefore,
we will focus now on showing that \eqref{eq:nbvp'} has a nontrivial 
solution for every $m>0$ given.

Suppose that $p^*=q=r$, where $r$ is the number defined in \eqref{eq:pq}. 
Then, \eqref{eq:nbvp'} is equivalent to \eqref{eq:bvp}.
Indeed, setting $u=(1/q-1/p)\vp$ in \eqref{eq:bvp}, we have
\[\vp'-(\vp')^2+\frac{p+q}{pq} \vp\vp''
+\frac{pq}{(p-q)^2}=0;\]
so that \eqref{eq:nbvp'} follows from $(p+q)/(pq)=1$ and $pq/(p-q)^2=m^2$.
Thus, solution \eqref{eq:sol1}, say $u_m$, of \eqref{eq:bvp} 
gives the solution of \eqref{eq:nbvp'} as $\vp=(1/q-1/p)^{-1}\,u_m
=2\sqrt{m^2+1/4}\,u_m$. 
\end{proof}

The graphs of solutions $\vp=2\sqrt{m^2+1/4}\,u_m$ in Corollary \ref{cor:nbvp}
are given in Figure \ref{fig:sol_m}
by using \texttt{InverseBetaRegularized} command of Wolfram Mathematica 11,
because $\sin_{p,q}^{-1}{x}$ can be written in terms of the incomplete beta function:
\[\sin_{p,q}^{-1}{x}=\frac{1}{q}\int_0^{x^q}s^{1/q-1}(1-s)^{1/p^*-1}\,ds.\]

\begin{figure}[htbp]
\begin{center}
\includegraphics[width=4cm,clip]{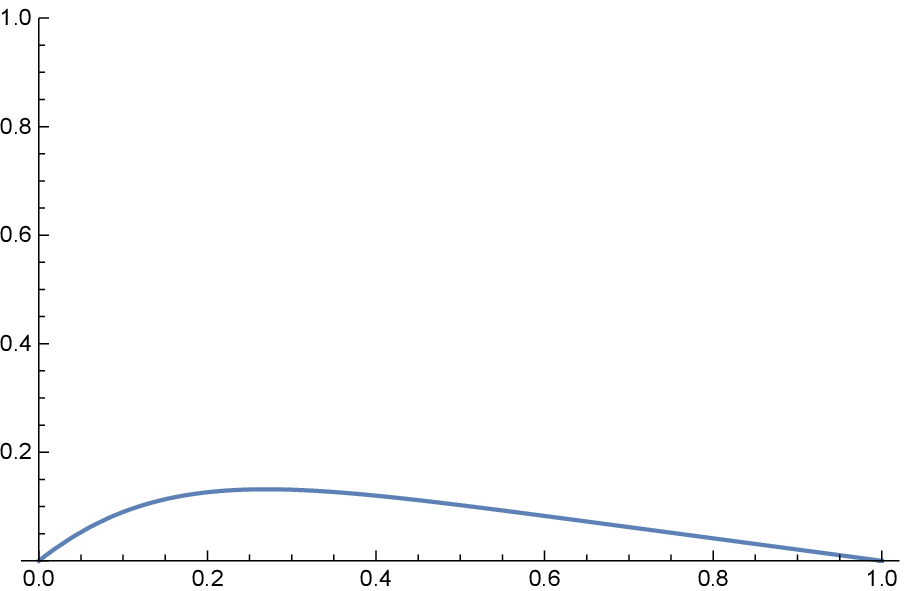}
\includegraphics[width=4cm,clip]{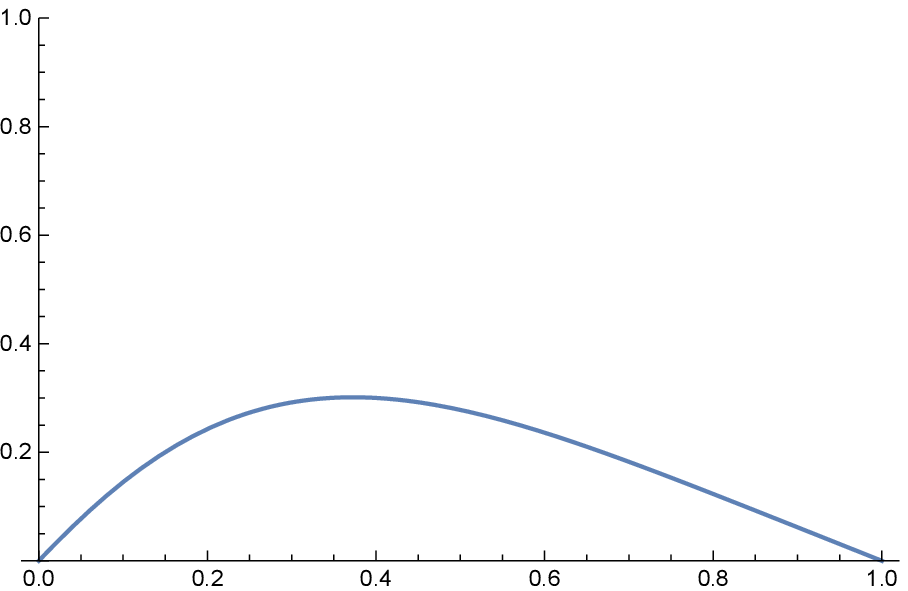}
\includegraphics[width=4cm,clip]{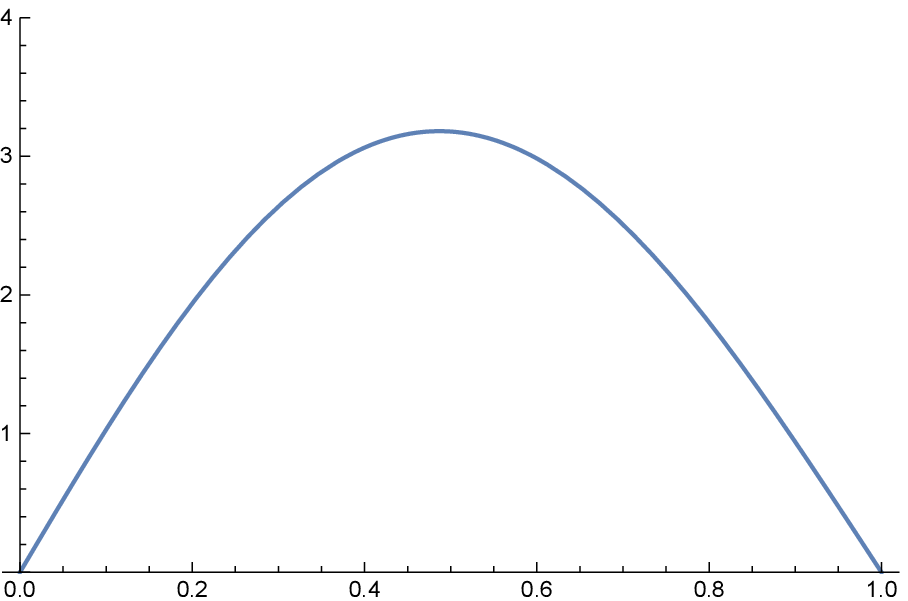}
\caption{Graphs of solutions of \eqref{eq:nbvp} with $H=1$ for $m=0.5,\ 1.0$ and $10.0$.}
\label{fig:sol_m}
\end{center}
\end{figure}

In case $p=q$, solution \eqref{eq:sol1} has a simple form. 

\begin{cor}
For $p \in (1,\infty)$, the positive solution of the boundary value problem
\begin{equation}
\label{eq:p=q}
-p^2(u')^2+2puu''+1=0, \quad u(0)=u(1)=0
\end{equation}
is
\[u=\frac{1}{p\pi_{2,p}}\sin_{2,p}{(\pi_{2,p}x)}.\]
Here, $\sin_{2,p}{(\pi_{2,p}x)}$ is defined in 
$[1/2,1]$ by $\sin_{2,p}{(\pi_{2,p}x)}=\sin_{2,p}{(\pi_{2,p}(1-x))}$.
In particular, the solution is symmetric with respect to $x=1/2$.
\end{cor}

\begin{proof}
This problem corresponds to \eqref{eq:bvp} with $H=1$ and $p=q$.
Then, by Theorem \ref{thm:CINT},
the positive solution is
\[u=\frac{2}{p\pi_{p^*,p}}\cos_{p^*,p}^{p^*-1}{\left(\frac{\pi_{p^*,p}}{2}x\right)}
\sin_{p^*,p}{\left(\frac{\pi_{p^*,p}}{2}x\right)}.\]
Moreover, the multiple-angle formula \cite[Theorem 1.1]{T2016b} for GTFs:
for $x \in [0,\pi_{p^*,p}/2]=[0,\pi_{2,p}/2^{2/p}]$, 
\[\sin_{2,p}{(2^{2/p}x)}=2^{2/p}\sin_{p^*,p}{x}\cos_{p^*,p}^{p^*-1}{x},\]
yields
\[u=\frac{1}{p\pi_{2,p}}\sin_{2,p}{(\pi_{2,p}x)}.\]
Thus the assertion follows.
\end{proof}

The graphs of solutions of \eqref{eq:p=q} are given in Figure \ref{fig:sol_p}.

\begin{figure}[htbp]
\begin{center}
\includegraphics[width=4cm,clip]{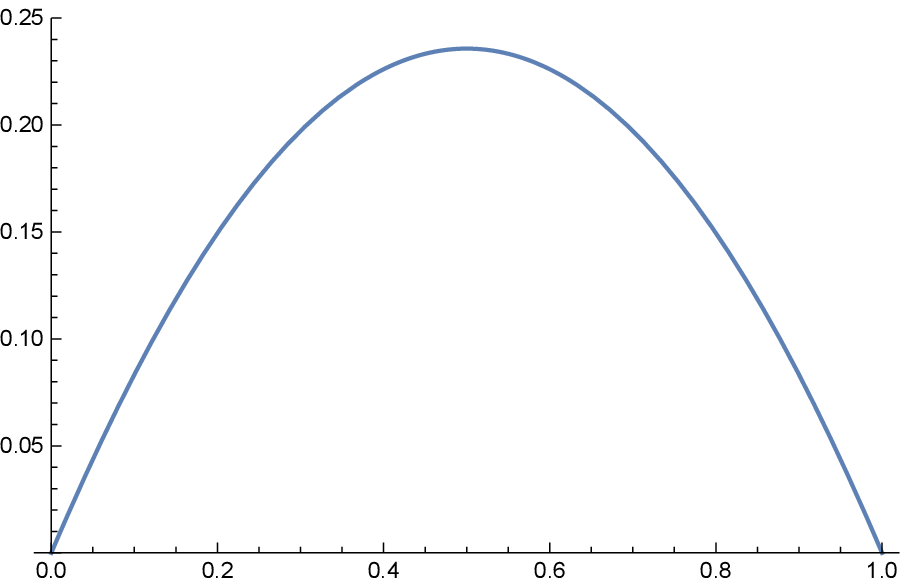}
\includegraphics[width=4cm,clip]{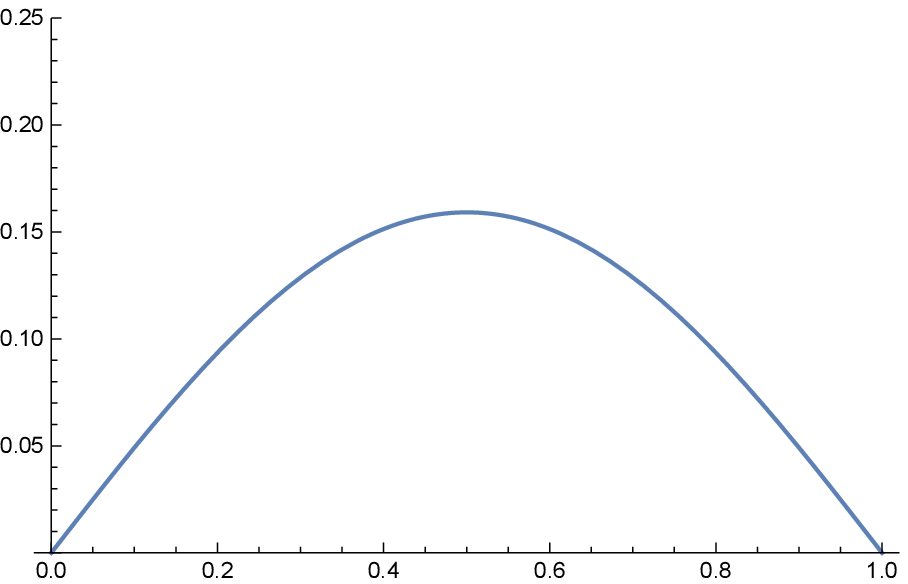}
\includegraphics[width=4cm,clip]{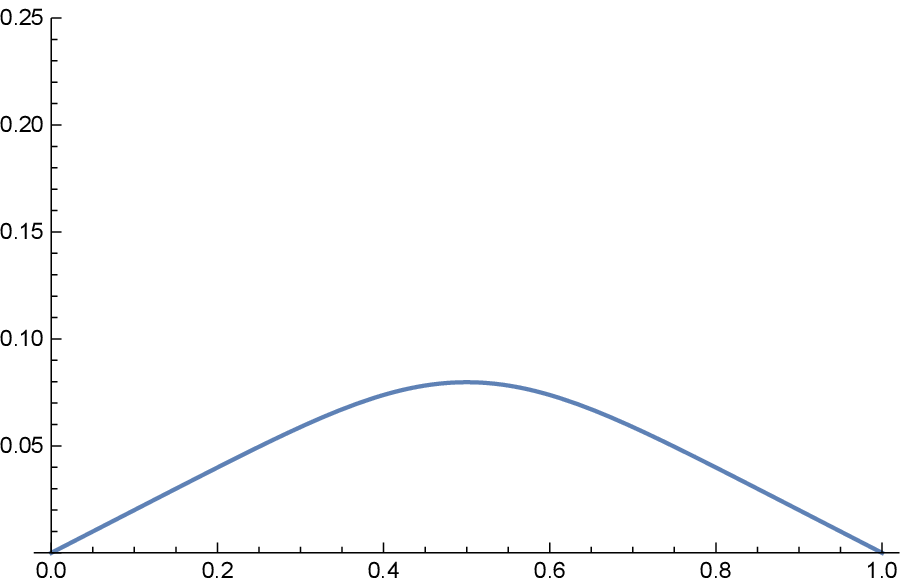}
\caption{Graphs of solutions of \eqref{eq:p=q} for $p=1.1,\ 2.0$ and $5.0$.}
\label{fig:sol_p}
\end{center}
\end{figure}

\section{Integral formulas and applications}

In this section we will give integral formulas involving primitive functions 
and Wallis-type formulas for GTFs. 
As applications, we obtain the counterparts for the lemniscate function $\slem{x}$ 
and the lemniscate constant $\varpi$. 

\begin{thm}
\label{thm:integration}
Let $p,\ q \in (1,\infty)$.
If $k>-1,\ \ell>1-p$ and $x \in [0, \pi_{p,q}/2]$, then
\begin{equation}
\label{eq:integration}
\int_0^x \sin_{p,q}^k{t} \cos_{p,q}^\ell{t}\,dt
=\frac{1}{k+1}\sin_{p,q}^{k+1}{x}\,
F\left(\frac{k+1}{q},\frac{1-\ell}{p};1+\frac{k+1}{q};\sin_{p,q}^q{x}\right).
\end{equation}
In particular,
\begin{equation}
\label{eq:definite}
\int_0^{\pi_{p,q}/2} \sin_{p,q}^k{t} \cos_{p,q}^\ell{t}\,dt
=\frac{1}{q}B\left(\frac{k+1}{q},1+\frac{\ell-1}{p}\right).
\end{equation}
\end{thm}

\begin{proof}
Letting $s=\sin_{p,q}^q{t}$,  we have
\begin{equation}
\label{eq:beta}
\int_0^x \sin_{p,q}^k{t} \cos_{p,q}^\ell{t}\,dt
=\frac{1}{q}\int_0^{\sin_{p,q}^q{x}} s^{(k+1)/q-1}(1-s)^{(\ell-1)/p}\,ds.
\end{equation}
Here, it is known that for $a,\ b>0$ and $0 \leq x \leq 1$, 
\[\int_0^x t^{a-1}(1-t)^{b-1}\,dt=\frac{x^a}{a}F(a,1-b;a+1;x)\]
(see \cite[6.6.8]{AS1964}, \cite[8.17.7]{OLBC2010} and \cite[p.41]{LE2011}).
Hence the right-hand side of \eqref{eq:beta} is 
\[\frac{1}{k+1}\sin_{p,q}^{k+1}{x}
F\left(\frac{k+1}{q},\frac{1-\ell}{p};1+\frac{k+1}{q};\sin_{p,q}^q{x}\right),\]
which implies \eqref{eq:integration}.

Next, letting $x=\pi_{p,q}/2$ in \eqref{eq:integration}, we obtain
\begin{equation}
\label{eq:teisekibun}
\int_0^{\pi_{p,q}/2} \sin_{p,q}^k{t} \cos_{p,q}^\ell{t}\,dt
=\frac{1}{k+1}F\left(\frac{k+1}{q},\frac{1-\ell}{p};1+\frac{k+1}{q};1\right).
\end{equation}
Here, it is also known that if $c>a+b$, then
\[F(a,b;c;1)=\frac{\Gamma(c)\Gamma(c-a-b)}{\Gamma(c-a)\Gamma(c-b)}\]
(see \cite[15.1.20]{AS1964}, \cite[15.4.20]{OLBC2010} and \cite[Theorem 2.2]{AAR1999}).
Here, $\Gamma(x)$ denotes the gamma function:
\[\Gamma(x)=\int_0^\infty e^{-t}t^{x-1}\,dt, \quad x>0.\]
It is well-known that $\Gamma(x+1)=x\Gamma(x)$ and 
$\Gamma(x)\Gamma(y)=\Gamma(x+y)B(x,y)$ for $x,\ y>0$.
Hence the right-hand side of \eqref{eq:teisekibun} is
\begin{align*}
\frac{1}{k+1}\frac{\Gamma(1+(k+1)/q) \Gamma (1-(1-\ell)/p)}
{\Gamma(1)\Gamma(1+(k+1)/q-(1-\ell)/p)}
&=\frac{1}{q}\frac{\Gamma((k+1)/q) \Gamma (1-(1-\ell)/p)}
{\Gamma(1+(k+1)/q-(1-\ell)/p)}\\
&=\frac{1}{q}B\left(\frac{k+1}{q},1+\frac{\ell-1}{p}\right),
\end{align*}
which implies \eqref{eq:definite}.
\end{proof}

\begin{rem}
When $\ell$ is a special value, the right-hand side of \eqref{eq:integration}
is a finite sum: for $k>-1$ and $n=0,1,2,\ldots$
\[\int_0^x \sin_{p,q}^k{t} \cos_{p,q}^{pn+1}{t}\,dt
=\sum_{m=0}^n \frac{(-1)^m}{k+1+qm} \binom{n}{m}\sin_{p,q}^{k+1+qm}{x}.\]
Indeed, the second parameter of $F$ in the 
right-hand side of \eqref{eq:integration} is $(1-\ell)/p=-n$,
and we see that $(-n)_m=(-1)^mm!\,\binom{n}{m}$ for $m \leq n$ 
and $(-n)_m=0$ for $m \geq n+1$. 
\end{rem}

\begin{cor}
\[\int_0^x \slem{t}\,dt
=\frac{1}{2}\sin^{-1}{(\slem^2{x})}.\]
In particular,
\[\int_0^{\varpi/2} \slem{t}\,dt
=\frac{\pi}{4}.\]
\end{cor}

\begin{proof}
Let $(p,q,k,\ell)=(2,4,1,0)$ in Theorem \ref{thm:integration},
and use $F(1/2,1/2;3/2;z^2)=z^{-1}\sin^{-1}z$.
\end{proof}

In the remainder of this section, we will construct the $(p,q)$-version of 
Wallis formulas. In what follows, we define
$1^*:=\infty,\ \pi_{s,1}:=2s^*$ and $\pi_{\infty,s}:=2$ for $s \in (1,\infty)$.

\begin{thm}
\label{thm:wallis_integration}
Let $p,\ q \in (1,\infty)$ and $n=0,1,2,\ldots$. Then, for $r \in (-1,q-1]$,
\begin{equation}
\label{eq:wallis_sin}
\int_0^{\pi_{p,q}/2} \sin_{p,q}^{qn+r}{t}\,dt
=\frac{u(1/u)_n}{q(1/p^*+1/u)_n}\frac{\pi_{p,u}}{2}, \quad \frac{1}{u}:=\frac{r+1}{q};
\end{equation}
for $r \in (1-p,1]$,
\begin{equation}
\label{eq:wallis_cos}
\int_0^{\pi_{p,q}/2} \cos_{p,q}^{pn+r}{t}\,dt
=\frac{(1/v)_n}{(1/v+1/q)_n}\frac{\pi_{v^*,q}}{2}, \quad \frac{1}{v}:=\frac{r+p-1}{p}.
\end{equation}
\end{thm}

\begin{proof}
We define $I_k$ and $J_\ell$ as
\begin{align}
\label{eq:I_k}
I_k&:=\int_0^{\pi_{p,q}/2}\sin_{p,q}^k{t}\,dt, \quad k>-1,\\
J_\ell &:=\int_0^{\pi_{p,q}/2}\cos_{p,q}^\ell{t}\,dt, \quad \ell>1-p. \notag
\end{align}
Let $k>q-1$. Then, \eqref{eq:dcos} yields
\begin{align*}
I_k
&=\int_0^{\pi_{p,q}/2} \sin_{p,q}^{k-q+1}{t}\cdot  \frac{d}{dt}\left(-\frac{p^*}{q}\cos_{p,q}^{p-1}{t}\right)\,dt\\
&=\left[-\frac{p^*}{q}\sin_{p,q}^{k-q+1}{t}\cos_{p,q}^{p-1}{t}\right]_0^{\pi_{p,q}/2}
+\frac{p^*}{q}(k-q+1)\int_0^{\pi_{p,q}/2} \sin_{p,q}^{k-q}{t}\cos_{p,q}^p{t}\,dt\\
&=\frac{p^*}{q}(k-q+1)(I_{k-q}-I_k).
\end{align*}
Therefore,
\[I_k=\frac{k-q+1}{q/p^*+k-q+1}I_{k-q}.\]
In particular, setting $k=qn+r,\ n=1,2,\ldots$ and $r \in (-1,q-1]$, we have
\begin{align*}
I_{qn+r}
&=\frac{q(n-1)+r+1}{q/p^*+q(n-1)+r+1}I_{q(n-1)+r}\\
&=\frac{n-1+1/u}{1/p^*+n-1+1/u}\frac{n-2+1/u}{1/p^*+n-2+1/u}
\cdots \frac{1/u}{1/p^*+1/u}I_r,
\end{align*}
where $1/u=(r+1)/q \in (0,1]$.
It follows from \eqref{eq:definite} that $I_r=u\pi_{p,u}/(2q)$. Thus,
\begin{align*}
I_{qn+r}
&=\frac{u(1/u)_n}{q(1/p^*+1/u)_n}\frac{\pi_{p,u}}{2}.
\end{align*}

In a similar way, for $\ell >1$ we obtain
\[J_\ell=\frac{\ell-1}{\ell-1+p/q}J_{\ell-p}.\]
Letting $\ell=pn+r,\ n=1,2,\ldots$ and $r \in (1-p,1]$, we get
$J_r=\pi_{v^*,q}/2$ and
\[J_{pn+r}=\frac{(1/v)_n}{(1/v+1/q)_n}\frac{\pi_{v^*,q}}{2},\]
where $1/v=(r+p-1)/p \in (0,1]$.
\end{proof}

\begin{cor}
\label{cor:wallis_integration}
Let $p,\ q \in (1,\infty)$ and $n=0,1,2,\ldots$. Then,
\begin{align}
\label{eq:qn}
\int_0^{\pi_{p,q}/2} \sin_{p,q}^{qn}{t}\,dt
&=\frac{(1/q)_n}{(1/p^*+1/q)_n}\frac{\pi_{p,q}}{2},\\
\int_0^{\pi_{p,q}/2} \sin_{p,q}^{qn+q-2}{t}\,dt
&=\frac{(1/q^*)_n}{(q-1)(1/p^*+1/q^*)_n}\frac{\pi_{p,q^*}}{2}, \notag\\
\int_0^{\pi_{p,q}/2} \sin_{p,q}^{qn+q-1}{t}\,dt
&=\frac{(1)_n}{(1/p^*+1)_n}\frac{p^*}{q}; \notag
\end{align}
and
\begin{align*}
\int_0^{\pi_{p,q}/2} \cos_{p,q}^{pn}{t}\,dt
&=\frac{(1/p^*)_n}{(1/p^*+1/q)_n}\frac{\pi_{p,q}}{2},\\
\int_0^{\pi_{p,q}/2} \cos_{p,q}^{pn+2-p}{t}\,dt
&=\frac{q(1/p)_n}{(1/p+1/q)_n}\frac{\pi_{p^*,q}}{2},\\
\int_0^{\pi_{p,q}/2} \cos_{p,q}^{pn+1}{t}\,dt
&=\frac{(1)_n}{(1+1/q)_n}.
\end{align*}
\end{cor}

\begin{proof}
The formulas for $\sin_{p,q}$ follow from \eqref{eq:wallis_sin} with
$r=0,\ q-2$ and $q-1$. 
The formulas for $\cos_{p,q}$ follow from \eqref{eq:wallis_cos} with
$r=0,\ 2-p$ and $1$. 
\end{proof}

\begin{rem}
Since $\sin_{p^*,p}{t}=\cos_{p^*,p}^{p^*-1}{(\pi_{p^*,p}/2-t)}$, 
shown in \cite{T2016b}, we see that for all $k>-1$
\[\int_0^{\pi_{p^*,p}/2} \sin_{p^*,p}^k{t}\,dt
=\int_0^{\pi_{p^*,p}/2} \cos_{p^*,p}^{(p^*-1)k}{t}\,dt.\]
\end{rem}

\begin{rem}
As in \cite{T2016c}, using the series expansion and the termwise integration with \eqref{eq:qn}, 
we can give the hypergeometric expansion of generalized complete elliptic integrals:
\begin{align*}
K_{p,q,r}(k)
&:=\int_0^{\pi_{p,q}/2} \frac{d\theta}{(1-k^q\sin_{p,q}^q{\theta})^{1/r}}
=\frac{\pi_{p,q}}{2}F\left(\frac{1}{q},\frac{1}{r};\frac{1}{p^*}+\frac{1}{q};k^q\right),\\
E_{p,q,r}(k)
&:=\int_0^{\pi_{p,q}/2} (1-k^q\sin_{p,q}^q{\theta})^{1/r^*}\,d\theta
=\frac{\pi_{p,q}}{2}F\left(\frac{1}{q},-\frac{1}{r^*};\frac{1}{p^*}+\frac{1}{q};k^q\right).
\end{align*} 
In addition, it is known that $K_{p,q,r}(k)$ and $E_{p,q,r}(k)$ satisfy \textit{Elliott's identity}:
\begin{multline*}
E_{p,q,r^*}(k)K_{p,r,q^*}(k')
+K_{p,q,r^*}(k)E_{p,r,q^*}(k')
-K_{p,q,r^*}(k)K_{p,r,q^*}(k')
=\frac{\pi_{p,q}\pi_{s,r}}{4},
\end{multline*}
where $k':=(1-k^q)^{1/r}$ and $1/s=1/p-1/q$. 
This is a generalization of Legendre's relation.
For more details we refer
the reader to \cite{T2016c} (in which the definition of $E_{p,q,r}(k)$ is slightly 
different from the above one) and the references given there.
\end{rem}

\begin{cor}
For $n=1,2,\ldots$,
\begin{align}
\int_0^{\pi/2} \sin^{2n}{t}\,dt
&=\int_0^{\pi/2} \cos^{2n}{t}\,dt
=\frac12\cdot\frac34\cdot\frac56\cdot\cdots\cdot\frac{2n-1}{2n}\cdot\frac{\pi}{2}, \notag \\
\int_0^{\pi/2} \sin^{2n+1}{t}\,dt
&=\int_0^{\pi/2} \cos^{2n+1}{t}\,dt
=\frac23\cdot\frac45\cdot\frac67\cdot\cdots\cdot\frac{2n}{2n+1},\notag \\
\int_0^{\varpi/2} \slem^{4n}{t}\,dt
&=\frac{1}{3}\cdot\frac{5}{7}\cdot\frac{9}{11}\cdot\cdots\cdot\frac{4n-3}{4n-1}\cdot\frac{\varpi}{2},\notag \\
\int_0^{\varpi/2} \slem^{4n+1}{t}\,dt
&=\frac{1}{2}\cdot\frac{3}{4}\cdot\frac{5}{6}\cdot\cdots\cdot\frac{2n-1}{2n}\cdot\frac{\pi}{4}, \label{eq:sl4n+1} \\
\int_0^{\varpi/2} \slem^{4n+2}{t}\,dt
&=\frac{3}{5}\cdot\frac{7}{9}\cdot\frac{11}{13}\cdot\cdots\cdot\frac{4n-1}{4n+1}\cdot\frac{\pi}{2\varpi}, \label{eq:sl} \\
\int_0^{\varpi/2} \slem^{4n+3}{t}\,dt
&=\frac{4}{7}\cdot\frac{8}{11}\cdot\frac{12}{15}\cdot\cdots\cdot\frac{4n}{4n+3}\cdot\frac{1}{2}.\notag
\end{align}
\end{cor}

\begin{proof}
In order to prove \eqref{eq:sl4n+1}, we apply \eqref{eq:wallis_sin} with $(p,q,r)=(2,4,1)$.
All the results, apart from \eqref{eq:sl4n+1}, come from Corollary \ref{cor:wallis_integration} 
for $(p,q)=(2,2)$ and $(2,4)$. In particular,
for \eqref{eq:sl} we obtain
\[\int_0^{\varpi/2} \slem^{4n+2}{t}\,dt
=\frac{3}{5}\cdot\frac{7}{9}\cdot\frac{11}{13}\cdot\cdots\cdot\frac{4n-1}{4n+1}\cdot \frac{\pi_{2,4/3}}{6}.\]
It suffices to show that
\[\pi_{2,4/3}=\frac{3\pi}{\varpi}.\]
By the symmetry of the beta function, $\pi_{2,4/3}=3\pi_{4,2}/2$. Moreover, the formula
$B(x,y)B(x+y,z)=B(y,z)B(y+z,x)$ with $(x,y,z)=(3/4,1/2,1/2)$ yields
$\pi_{4,2}B(5/4,1/2)=4\pi/3$. Here, the formula
$(x+y)B(1+x,y)=xB(x,y)$ with $(x,y)=(1/4,1/2)$
gives
$B(5/4,1/2)=2\pi_{2,4}/3=2\varpi/3$; so that $\pi_{4,2}=2\pi/\varpi$.
Consequently, we conclude $\pi_{2,4/3}=3\pi/\varpi$.
\end{proof}

\begin{thm}
\label{thm:wallis_product_general}
Let $p,\ q \in (1,\infty)$. Then,
\[\frac{\pi_{p,q}}{2}
=\prod_{n=1}^\infty \left(1-\frac{1}{pn(qn+1-q/p)}\right)^{-1}.\]
\end{thm}

\begin{proof}
It follows from \eqref{eq:I_k} and Corollary \ref{cor:wallis_integration} that
\begin{equation}
\label{eq:ratio}
\frac{q\pi_{p,q}}{2p^*} \frac{I_{qn+q-1}}{I_{qn}}
=\frac{(1)_n(1/p^*+1/q)_n}{(1/q)_n(1/p^*+1)_n}
=\frac{(1)_n(1/p^*+1/q)_n}{(1+1/q)_n(1/p^*)_n}\frac{qn+1}{p^*n+1}.
\end{equation}
Now, since $0<I_{qn+q-1}<I_{qn}<I_{qn-1}$,  we have
\[1<\frac{I_{qn}}{I_{qn+q-1}}<\frac{I_{q(n-1)+q-1}}{I_{qn+q-1}}.\]
Moreover, 
\[\lim_{n \to \infty} \frac{I_{q(n-1)+q-1}}{I_{qn+q-1}}
=\lim_{n \to \infty} \frac{1/p^*+n}{n}=1;\]
so that 
\[\lim_{n \to \infty} \frac{I_{qn}}{I_{qn+q-1}}=1.\]
Thus, \eqref{eq:ratio} yields
\[\frac{\pi_{p,q}}{2}=\lim_{n \to \infty}\frac{(1)_n(1/p^*+1/q)_n}{(1+1/q)_n(1/p^*)_n};\]
that is,
\begin{align*}
\frac{\pi_{p,q}}{2}
&=\prod_{n=1}^\infty \frac{pn(qn+1-q/p)}{(pn-1)(qn+1)}\\
&=\prod_{n=1}^\infty \left(1-\frac{1}{pn(qn+1-q/p)}\right)^{-1},
\end{align*}
and the proof is complete.
\end{proof}

\begin{cor}
\[\frac{\varpi}{2}
=\prod_{n=1}^\infty \left(1-\frac{1}{2n(4n-1)}\right)^{-1}.\]
\end{cor}

\begin{proof}
This is Theorem \ref{thm:wallis_product_general} for $(p,q)=(2,4)$.
\end{proof}

\begin{rem}
A similar formula for $\varpi$ is obtained in \cite[Theorem 3.3]{H2014}.
\end{rem}

\section*{Appendix}

For the convenience of the reader we repeat formulas in \cite[Proposition 3.2]{EGL2012},
thus making our exposition self-contained: for $p,\ q \in (1,\infty)$ and $x \in [0,1]$
\begin{align}
\sin_{p,q}{\left( \frac{\pi_{p,q}}{2}x\right)}
&=\cos_{q^*,p^*}^{q^*-1}{\left( \frac{\pi_{q^*,p^*}}{2}(1-x)\right)}, \label{eq:sinpq}\\
\cos_{p,q}{\left( \frac{\pi_{p,q}}{2}x\right)}
&=\sin_{q^*,p^*}^{p^*-1}{\left( \frac{\pi_{q^*,p^*}}{2}(1-x)\right)}. \label{eq:cospq}
\end{align}

\begin{proof}
Formula \eqref{eq:sinpq} follows immediately from \eqref{eq:cospq} by 
replacing $(p,q,x)$ to $(q^*,p^*,1-x)$. Therefore, we will prove \eqref{eq:cospq}.

Let $y=(\cos_{p,q}^{p-1})^{-1}{t}$ for $t \in [0,1]$. 
By \eqref{eq:dcos},
the derivative of the inverse function is
\[\frac{dy}{dt}=-\frac{p^*}{q}\frac{1}{(1-t^{p^*})^{1/q^*}}.\]
Integrating both-sides from $t$ to $1$, we have
\[y=\frac{p^*}{q} \int_t^{1} \frac{dt}{(1-t^{p^*})^{1/q^*}}.\]
Thus,
\[ \sin_{q^*,p^*}^{-1}{t}+\frac{q}{p^*}y=\int_0^{1} \frac{dt}{(1-t^{p^*})^{1/q^*}}= \frac{\pi_{q^*,p^*}}{2},\]
which means
\[t=\sin_{q^*,p^*}{\left(\frac{\pi_{q^*,p^*}}{2}-\frac{q}{p^*}y\right)}.\]
Since $y \in [0,\pi_{p,q}/2]$, we can write $y=(\pi_{p,q}/2)x$ for $x \in [0,1]$.
Moreover, by $\pi_{p,q}=(p^*/q)\pi_{q^*,p^*}$, we obtain
\[\cos_{p,q}^{p-1}{\left( \frac{\pi_{p,q}}{2}x\right)}
=\sin_{q^*,p^*}{\left(\frac{\pi_{q^*,p^*}}{2}(1-x)\right)}.\]
This is equivalent to \eqref{eq:cospq}. 
\end{proof}

\section*{Acknowledgments}
The authors would like to thank the anonymous reviewers for his/her valuable comments
and suggestions to improve the quality of the paper.
Also, the authors would like to thank Professor Okihiro Sawada for informing problem
\eqref{eq:nbvp} and paper \cite{CINT2015}.

\end{document}